\documentclass[11pt]{article}

\usepackage{amsfonts,amssymb,amsmath,amsthm,latexsym,color,epsfig}
\usepackage[textsize=scriptsize,colorinlistoftodos]{todonotes}
\theoremstyle{plain}
\newtheorem{thm}{Theorem}[section]
\newtheorem{prop}[thm]{Proposition}
\newtheorem{lem}[thm]{Lemma}
\newtheorem{claim}[thm]{Claim}
\newtheorem{ques}[thm]{Question}

\newcommand{\QS}{\textrm{`}}
\newcommand{\QT}{\textrm{'}}

\title{\vspace{-0.7cm}The Erd\H{o}s-Gy\'arf\'as problem on generalized Ramsey numbers}
\author{David Conlon \thanks{Mathematical Institute, Oxford OX2 6GG,
United Kingdom. Email: {\tt david.conlon@maths.ox.ac.uk}. Research
supported by a Royal Society University Research Fellowship.}
\and
Jacob Fox \thanks{Department of Mathematics, MIT, Cambridge,
MA 02139-4307. Email: {\tt fox@math.mit.edu}. Research supported by
a Packard Fellowship, by a Simons Fellowship, by NSF grant DMS-1069197, by an Alfred P. Sloan Fellowship  and by an MIT NEC Corporation Award.}
\and
Choongbum Lee \thanks{Department of Mathematics,
MIT, Cambridge, MA 02139-4307. Email: {\tt cb\_lee@math.mit.edu}.}
\and
Benny Sudakov \thanks{Department of Mathematics, ETH, 8092 Zurich, Switzerland and
Department of Mathematics, UCLA, Los Angeles,
CA 90095. Email: {\tt benjamin.sudakov@math.ethz.ch}. Research supported in
part by SNSF grant 200021-149111 and by a USA-Israel BSF grant.}}
\date{}

\begin{document}
\maketitle

\begin{abstract} Fix positive integers $p$ and $q$ with $2 \leq q \leq {p \choose 2}$. 
An edge-coloring of the complete graph $K_n$ is said to be a $(p, q)$-coloring if every $K_p$ receives at least $q$ different colors. The function $f(n, p, q)$ is the minimum number of colors that are needed for $K_n$ to have a $(p,q)$-coloring. This function was introduced by Erd\H{o}s and Shelah about 40 years ago, but Erd\H{o}s and Gy\'arf\'as were the first to study the function in  a systematic way. They proved that $f(n, p, p)$ is polynomial in $n$ and asked to determine the maximum $q$, depending on $p$, for which $f(n,p,q)$ is subpolynomial in $n$. We prove that the answer is $p-1$. 
\end{abstract}

\section{Introduction}

The Ramsey number $r_k(p)$ is the smallest natural number $n$ such that every $k$-coloring of the edges of the complete graph $K_n$ contains a monochromatic $K_p$. The existence of $r_k(3)$ was first shown by Schur \cite{Sc} in 1916 in his work on Fermat's Last Theorem and it is known that $r_k(3)$ is at least exponential in $k$ and at most a multiple of $k!$. It is a central problem in graph Ramsey theory to close the gap between the lower and upper bound, with connections to various problems in combinatorics, geometry, number theory, theoretical computer science and information theory (see, e.g., \cite{NeRo, Ro}).

The following natural generalization of the Ramsey function was first introduced by Erd\H{o}s and Shelah \cite{E75, E81} and studied in depth by Erd\H{o}s and Gy\'arf\'as \cite{EG97}. Let $p$ and $q$ be positive integers with $2 \leq q \leq \binom{p}{2}$. An edge-coloring of the complete graph $K_n$ is said to be a $(p, q)$-coloring if every $K_p$ receives at least $q$ different colors. The function $f(n, p, q)$ is the minimum number of colors that are needed for $K_n$ to have a $(p,q)$-coloring. 

To see that this is indeed a generalization of the usual Ramsey function, note that $f(n, p, 2)$ is the minimum number of colors needed to guarantee that no $K_p$ is monochromatic. That is, $f(n, p, 2)$ is the inverse of the Ramsey function $r_k(p)$ and so we have
\[c' \frac{\log n}{\log \log n} \leq f(n, 3, 2) \leq c \log n.\]

Erd\H{o}s and Gy\'arf\'as \cite{EG97} proved a number of interesting results about the function $f(n,p,q)$, demonstrating how the function falls off from being equal to $\binom{n}{2}$ when $q = \binom{p}{2}$ to being at most logarithmic when $q = 2$. In so doing, they determined ranges of $p$ and $q$ where the function $f(n, p, q)$ is linear in $n$, where it is quadratic in $n$ and where it is asymptotically equal to $\binom{n}{2}$. Many of these results were subsequently sharpened by S\'ark\"ozy and Selkow \cite{SS01, SS03}.

One simple observation made by Erd\H{o}s and Gy\'arf\'as is that $f(n, p, p)$ is always polynomial in $n$. To see this, it is sufficient to note that if a coloring uses fewer than $n^{1/(p-2)} - 1$ colors then it necessarily contains a $K_p$ which uses at most $p - 1$ colors. For $p = 3$, this is easy to see since one only needs that some vertex has at least two neighbors in the same color.  For $p = 4$, we have that any vertex will have at least $n^{1/2}$ neighbors in some fixed color. But, since there are fewer than $n^{1/2} - 1$ colors on this neighborhood of size at least $n^{1/2}$, the case $p = 3$ implies that it contains a triangle with at most two colors.  The general case follows similarly. 

Erd\H{o}s and Gy\'arf\'as \cite{EG97} asked whether this result is best possible, that is, whether $q = p$ is the smallest value of $q$ for which $f(n, p, q)$ is polynomial in $n$. For $p = 3$, this is certainly true, since we know that $f(n, 3, 2) \leq c \log n$.  However, for general $p$, they were only able to show that $f(n, p, \lceil \log p \rceil)$ is subpolynomial, where here and throughout the paper we use $\log$ to denote the logarithm taken base $2$. This left the question of determining whether $f(n, p, p - 1)$ is subpolynomial wide open, even for $p = 4$.

The first progress on this question was made by Mubayi \cite{M98}, who found an elegant construction which implies that $f(n, 4, 3) \leq e^{c \sqrt{\log n}}$. This construction was also used by Eichhorn and Mubayi \cite{EM00} to demonstrate that $f(n, 5, 4) \leq e^{c \sqrt{\log n}}$. More generally, they used the construction to show that $f(n, p, 2 \lceil \log p \rceil - 2)$ is subpolynomial for all $p \geq 5$.

In this paper, we answer the question of Erd\H{o}s and Gy\'arf\'as in the positive for all $p$. That is, we prove that $f(n, p, p -1)$ is subpolynomial for all $p$. Quantitatively, our main theorem is the following.

\begin{thm} \label{thm:main}
For all natural numbers $p \geq 4$ and $n \geq 1$, 
\[f(n, p, p-1) \leq 2^{16p (\log n)^{1 - 1/(p-2)} \log \log n}.\]
\end{thm}

In Section \ref{sec:definition}, we define our $(p,p-1)$-coloring
by a recursive procedure. We begin by reviewing Mubayi's $(4,3)$-coloring,
as it is the base case of our recursion. The formal proof
of the fact that our coloring is indeed a $(p,p-1)$-coloring is quite
technical and thus we first give an outline of the proof 
in Section \ref{sec:outline}. Then, in Section \ref{sec:properties}, we 
establish some properties of the coloring. Finally, in Section \ref{sect:proofmainthm},
we prove that the coloring given in Section \ref{sec:definition} is a
$(p,p-1)$-coloring. We will conclude with some further remarks.

\medskip{}

\textbf{Notation.} For vectors $v\in X^{t_{1}+t_{2}},v_{1}\in X^{t_{1}},v_{2}\in X^{t_{2}}$,
we will often use the notation 
\[
v=(v_{1},v_{2}),
\]
in order to indicate that the $i$-th coordinate of $v$ is equal to the
$i$-th coordinate of $v_{1}$ for $1\le i\le t_{1}$ and the
$(t_{1}+j)$-th coordinate of $v$ is equal to the $j$-th coordinate
of $v_{2}$ for $1\le j\le t_{2}$. We will use similar notation for several
vectors. Throughout the paper, $\log$ denotes the base 2 logarithm.
For the sake of clarity of presentation, we systematically omit floor
and ceiling signs whenever they are not essential. 

\section{The coloring construction}
\label{sec:definition}

The purpose of this section is to define the coloring used to prove Theorem \ref{thm:main}. The coloring can be considered as a generalization of (a variant of) Mubayi's $(4,3)$-coloring. We therefore first introduce this coloring and then redefine it in a way that can be naturally extended. We then present the coloring used to prove Theorem \ref{thm:main}. As it is a rather involved recursive definition, we give an example to illustrate it. We conclude the section by establishing a bound on the number of colors used in this coloring. In the following sections, we will show that this coloring is a $(p,p-1)$-coloring, completing the proof. 

\subsection{Mubayi's $(4,3)$-coloring\label{sub:Mubayi}}

Let $N=m^{t}$ for some integers $m$ and $t$. Suppose that we are
given two distinct vectors $v,w\in[m]^{t}$ of the form $v=(v_{1},\ldots,v_{t})$
and $w=(w_{1},\ldots,w_{t})$. Define 
\[
c(v,w)=\big(\{v_{i},w_{i}\},a_{1},\ldots,a_{t}\big),
\]
where $i$ is the least coordinate in which $v_{i}\neq w_{i}$ and
$a_{j}=0$ if $v_{j}=w_{j}$ and $a_{j}=1$ if $v_{j}\neq w_{j}$.
If $v=w$, define
\[
c(v,v)=0.
\]
Note that $c$ is a symmetric function.
This is a variant of Mubayi's coloring and can be proved to be a
$(p,p-1)$-coloring for small values of $p$. 

One might suspect that this is a $(p,p-1)$-coloring for large integers
$p$ as well, but, unfortunately, it fails to
be a $(26,25)$-coloring (and a $(p,p-1)$-coloring for all $p\ge26$)
for the following reason. Consider the set $\{1,2,3\}^{3}$. This
set has $3^{3}=27$ elements and at most $3\cdot2^{3}=24$
colors are used in coloring this set. Therefore, we can find 26 vertices
with at most 24 colors within the set. Moreover, for every fixed $p$
and large enough $N$, letting $s=\sqrt{\log p}$, the set $S=\{1,2,\ldots,2^{s}\}^{s}$
has cardinality $2^{s^{2}}=p$ and uses at most ${2^{s} \choose 2}2^{s}<2^{3s}=2^{3\sqrt{\log p}}$
colors and, for large enough $m$ and $t$, is a subset of $[m]^{t}$.
Hence, this edge-coloring of the complete graph on $[N]$ fails
to be a $(p,2^{3\sqrt{\log p}})$-coloring.

\subsection{Redefining Mubayi's coloring}

Before proceeding further, let us redefine the coloring given above
from a slightly different perspective. We do this to motivate the $(p,p-1)$-coloring which we use to establish Theorem \ref{thm:main}. Let $m=2^{r_{1}}$ and, abusing notation, identify the set $[m]$
with $\{0,1\}^{r_{1}}$. Let $r_{2}=r_{1}t$ for some positive integer
$t$. Suppose that we are given two vectors $v,w\in[m]^{t}=\{0,1\}^{r_{1}t}$. We decompose
$v$ as $v=(v_{1}^{(1)},\ldots,v_{t}^{(1)})$, where $v_{i}^{(1)}\in\{0,1\}^{r_{1}}$
for $i=1,2,\ldots,t$ and similarly decompose $w$. The function $c$
was defined as follows:
\[
c(v,w)=\big(\{v_{i}^{(1)},w_{i}^{(1)}\},a_{1},\ldots,a_{t}\big),
\]
where $i$ is the least coordinate in which $v_{i}^{(1)}\neq w_{i}^{(1)}$ and,
for $j=1,2,\ldots,t$, $a_j$ represents whether $v_{j}^{(1)}=w_{j}^{(1)}$
or not. If $v=w$, then $c(v,v)=0$. 

Define $h_{1}$ as the first coordinate of $c$. That is, $h_{1}(v,w)=\{v_{i}^{(1)},w_{i}^{(1)}\}$
(we let $h_{1}(v,v)=0$ for convenience). Note that $h_{1}$ takes
a pair of vectors of length $r_{2}=r_{1}t$ as input and outputs
a pair of vectors of length $r_{1}$.

For two vectors $x,y\in\{0,1\}^{r_{1}}$ of the form $x=(x_{1},\ldots,x_{r_{1}})$,
$y=(y_{1},\ldots,y_{r_{1}})$, define the function
$h_{0}$ as follows. We have $h_{0}(x,x)=0$ for each $x$ and, if $x \neq y$,
then $h_{0}(x,y)=\{x_{i},y_{i}\},$ where $i$ is the minimum
index for which $x_{i}\neq y_{i}$. Since all $x_{i}$ and $y_{i}$
are either 0 or 1, there are only two possible outcomes for $h_{0}$,
$0$ if the two vectors are equal and $\{0,1\}$ if they are not equal.
Note that $h_{0}$ takes a pair of 
vectors of length $r_{1}$ as input and outputs a pair of vectors 
of length $r_{0}=1$. Thus, both $h_{1}$ and $h_{0}$ are functions
which record the first `block' that is different. The difference between the
two functions lies in their interpretation of `block': for $h_{1}$ it is
a subvector of length $r_{1}$ and for $h_{0}$ it is a subvector of 
length $r_{0}$.

Summarizing, we see that $c$ is equivalent to the coloring $c'$ given by
\[
c'(v,w)=\Big(h_{1}(v,w),h_{0}(v_{1}^{(1)},w_{1}^{(1)}),\ldots,h_{0}(v_{t}^{(1)},w_{t}^{(1)})\Big).
\]
Informally, we first decompose the given pair of vectors $v$ and
$w$ into subvectors of length $r_{2}$ and apply $h_{1}$ (we observe
only a single subvector in this case since $v$ and $w$ themselves
are vectors of length $r_{2}$). Then we decompose $v$ and $w$
into subvectors of length $r_{1}$ and apply $h_{0}$ to each 
corresponding pair of subvectors of $v$ and $w$.

\subsection{Definition of the coloring}\label{subsect:defcolor}

In this section, we generalize the construction given in the previous section to obtain a $(p,p-1)$-coloring. 

For a positive integer $\alpha$, we will describe the coloring as
an edge-coloring of the complete graph over the vertex set $\{0,1\}^{\alpha}$.
Let $r_{0},r_{1},\ldots$ be a sequence of positive integers such that $r_{0}=1$
and $r_{d-1}$ divides $r_{d}$ for all $d\ge1$. 

For a set of indices $I$, let $\pi_{I}$ be the canonical projection
map from $\{0,1\}^{\alpha}$ to $\{0,1\}^{I}$. We will write $\pi_{i}$
instead of $\pi_{[i]}$ for convenience. Thus $\pi_{i}$ is the projection
map to the first $i$ coordinates. 

The key idea in the construction is to understand vectors at several
different resolutions. Suppose that we are given two vectors $v,w\in\{0,1\}^{\alpha}$.
For $d\ge0$, let $a_d$ and $b_d$ be integers
satisfying $a_{d}\ge0$ and $1\le b_{d}\le r_{d}$ such that $\alpha=a_{d}r_{d}+b_{d}$.
Let 
\[
v=\Big(v_{1}^{(d)},v_{2}^{(d)},\ldots,v_{a_{d}+1}^{(d)}\Big),
\]
where $v_{i}^{(d)}\in\{0,1\}^{r_{d}}$ for $i=1,2,\ldots,a_{d}$
and $v_{a_{d}+1}^{(d)}\in\{0,1\}^{b_{d}}$. We refer to the vectors
$v_{i}^{(d)}$ as \emph{blocks of resolution d}. We similarly decompose
$w$ as $w=\big(w_{1}^{(d)},w_{2}^{(d)},\ldots,w_{a_{d}}^{(d)},w_{a_{d}+1}^{(d)}\big)$
for $d\ge0$.

We first define two auxiliary families of functions $\eta_{d}$ and
$\xi_{d}$. For $d\ge0$, if $v\neq w$, let 
\[
\eta_{d}(v,w)=\Big(i,\{v_{i}^{(d)},w_{i}^{(d)}\}\Big),
\]
where $i$ is the minimum index such that $v_{i}^{(d)}\neq w_{i}^{(d)}$.
If $v=w$, let 
\[
\eta_{d}(v,v)=0.
\]
Note that $\eta_{d}$ is a symmetric function.
Further note that $\eta_{d}$ is slightly different from $h_{d}$ defined
in the previous subsection since we add an additional coordinate which
records the index $i$ as well. The main theorem is valid even if
we do not add this index, but we choose to add it as it simplifies
the proof. We refer the reader to Subsection \ref{sec:conclude_fewer_colors} for 
a further discussion of this point.

For $d\ge0$, let 
\[
\xi_{d}(v,w)=\Big(\eta_{d}\big(v_{1}^{(d+1)},w_{1}^{(d+1)}\big),\ldots,\eta_{d}\big(v_{a_{d+1}+1}^{(d+1)},w_{a_{d+1}+1}^{(d+1)}\big)\Big).
\]
Note that the function $\xi_{d}$ decomposes the vectors into blocks
of resolution $d+1$ and outputs a vector containing
information about blocks of resolution $d$.

For $d\ge0$, let 
\[
c_{d}=\xi_{d}\times\xi_{d-1}\times\ldots\times\xi_{0}.
\]
Note that the coloring $c_{d}$ depends on the choice of the parameters
$r_{0},r_{1},\ldots,r_{d+1}$. 

\medskip

We prove our main theorem in two steps: we first estimate the number
of colors and then prove that it is a $(p,p-1)$-coloring.

\begin{thm}
\label{thm:main1}Let $p$ and $\beta$ be fixed positive integers with $\beta \not = 1$. For the choice $r_{i}=\beta^{i}$
for $0\le i\le p+1$, the edge-coloring $c_{p}$ of the complete graph
on $n=2^{\beta^{p+1}}$ vertices uses at most $2^{4(\log n)^{1-1/(p+1)}\log\log n}$ colors.
\end{thm}

\begin{thm}
\label{thm:main2} Let $p$ and $\alpha$ be fixed positive integers.
Then, for every choice of parameters $r_{1},\ldots,r_{p+1}$, the edge-coloring
$c_{p}$ is a $(p+3,p+2)$-coloring of the complete graph on the vertex
set $\{0,1\}^{\alpha}$.
\end{thm}
For integers $n$ of the form $n=2^{\beta^{p+1}}$, Theorem \ref{thm:main}
follows from Theorems \ref{thm:main1} and \ref{thm:main2}. For general
$n \ge p+3 \ge 4$, first notice that if $n^2 < 2^{16p(\log n)^{1-1/(p+1)}\log \log n}$, then
the statement is trivially true, as we may color each edge with
different colors. Hence, we may assume that the inequality does not hold,
from which it follows that 
\[ 2\log n \ge 16p(\log n)^{1-1/(p+1)}\log \log n \ge 16p(\log n)^{1-1/(p+1)}\] 
and $n \ge 2^{(8p)^{p+1}}$. Hence, there exists an integer of the form $2^{\beta^{p+1}}$
which is at most $n^{(1+1/8p)^{p+1}} \le n^2$. Therefore, there exists a $(p+3,p+2)$-coloring
of the complete graph on the vertex set $[n]$ using at most 
\[
2^{4(2\log n)^{1-1/(p+1)}\log(2\log n)}
\le 2^{4 \cdot 2(\log n)^{1-1/(p+1)} (1+\log\log n)}
\le2^{16(\log n)^{1-1/(p+1)}\log\log n}
\]
 colors (in the second inequality we used the fact that $\log\log n\ge\log \log 4 \ge1$).
 Thus we obtain Theorem \ref{thm:main}. Theorem \ref{thm:main1} is proved in Subsection \ref{subsect:numbcolors}, while Theorem \ref{thm:main2} is proved in Section \ref{sect:proofmainthm} and  builds on the two sections leading up to it.

\subsection{Example}

Let us illustrate the coloring by working out a small example. Suppose
that $r_{1}=2$ and $r_{2}=4$. Let $v=(0,0,1,0,1,1,0)$ and $w=(0,0,1,1,1,0,0)$
be vectors in $\{0,1\}^{7}$. Then
\begin{align*}
v & =(0,0,1,0,1,1,0)=\Big(\QS 0,0 \QT, \QS 1,0 \QT, \QS 1,1 \QT, \QS 0 \QT\Big)
                    =\Big(\QS 0,0,1,0 \QT,\QS 1,1,0 \QT\Big),
\end{align*}
where the quotation marks indicate the blocks of each
resolution. Similarly,
\begin{align*}
w & =(0,0,1,1,1,0,0)=\Big(\QS 0,0\QT,\QS 1,1\QT,\QS 1,0\QT,\QS 0\QT\Big)=\Big(\QS 0,0,1,1\QT,\QS 1,0,0\QT\Big).
\end{align*}
The function $\eta_{0}$ records the first pair of blocks of resolution
0 which are different. So 
\[
\eta_{0}(v,w)=(4,\{0,1\}),
\]
where the value of the first coordinate, $4$, indicates that $v$
and $w$ first differ in the fourth coordinate. Similarly, the function
$\eta_{1}$ will record the first pair of blocks of resolution 1 which
are different. So 
\[
\eta_{1}(v,w)=\Big(2,\{(1,0),(1,1)\}\Big).
\]

Computing $\xi_{0}$ and $\xi_{1}$ involves one more step. To compute
$\xi_{0}$, we apply $\eta_{0}$ to each pair of blocks of resolution
1. Therefore, 
\begin{align*}
\xi_{0}(v,w) & =\Big(\eta_{0}\big((0,0),(0,0)\big),\eta_{0}\big((1,0),(1,1)\big),\eta_{0}\big((1,1),(1,0)\big),\eta_{0}\big((0),(0)\big)\Big)\\
 & =\big(0,(2,\{0,1\}),(2,\{1,0\}),0\big),
\end{align*}
which is a vector of length four. 

Similarly, to compute $\xi_{1}$,
we apply $\eta_{1}$ to each pair of blocks of resolution 2. Therefore,
\begin{align*}
\xi_{1}(v,w) & =\Big(\eta_{1}\big((0,0,1,0),(0,0,1,1)\big),\eta_{1}\big((1,1,0),(1,0,0)\big)\Big)\\
 & =\left(\Big(2,\big\{(1,0),(1,1)\big\}\Big),\Big(1,\big\{(1,1),(1,0)\big\}\Big)\right),
\end{align*}
which is a vector of length two. 

\subsection{Number of colors}\label{subsect:numbcolors}

In this subsection, we prove Theorem \ref{thm:main1}.

\begin{proof}[Proof of Theorem \ref{thm:main1}]

Recall that $\beta$ is a positive integer greater than $1$ and
$r_{d}=\beta^{d}$ for $0\le d\le p+1$. Let $\alpha=\beta^{p+1}$.
The goal here is to give an upper bound on the number of colors in the edge-coloring $c_p$ of the complete graph with vertex set $\{0,1\}^{\alpha}=\{0,1\}^{\beta^{p+1}}$.
First, for $0\le d\le p$, the function $\eta_{d}$ outputs either
zero or an index and a pair of distinct blocks of resolution $d$. Hence, there
are at most $1+\alpha\cdot2^{r_{d}}(2^{r_d} -1) \le \alpha2^{2\beta^{d}}$ possible
outcomes for the function $\eta_{d}$. Second, for $0\le d\le p$,
the function $\xi_{d}$ is a product of $\frac{\alpha}{r_{d+1}}=\beta^{p-d}$
outcomes of $\eta_{d}$. Hence, there are at most 
\[
\big(\alpha\cdot2^{2\beta^{d}}\big)^{\beta^{p-d}}
= \beta^{(p+1)\beta^{p-d}}\cdot2^{2\beta^{p}}
\]
possible outcomes for the function $\xi_{d}$. Since $c_{p}$ is defined
as $\xi_{p}\times\xi_{p-1}\times\cdots\times\xi_{0}$, the total number
of colors used in $c_{p}$ is at most 
\begin{align*}
\prod_{d=0}^{p} \left(\beta^{(p+1)\beta^{p-d}}\cdot2^{2\beta^{p}}\right)
\le \beta^{2(p+1)\beta^{p}} 2^{2(p+1)\beta^p} \le 2^{4(p+1)\beta^p\log \beta}.
\end{align*}
Let $n=2^{\alpha}=2^{\beta^{p+1}}$ and note that $\beta^{p}=(\log n)^{1-1/(p+1)}$
and $\log\beta=\frac{1}{p+1}\log\log n$. Thus, we have colored the edges
of the complete graph on $n$ vertices using at most 
\[
2^{4(\log n)^{1-1/(p+1)}\log\log n}
\]
colors, as claimed in Theorem \ref{thm:main1}.
\end{proof}

\medskip{}

As we saw in Subsection \ref{sub:Mubayi}, for large enough $q$,
Mubayi's coloring (which is similar to $c_{1}$) is not a $(q,q-1)$-coloring
or even a $(q,q^{\varepsilon})$-coloring for any fixed $\varepsilon>0$.
Similarly, we can see that the same is true for the coloring $c_{p}$ 
for every fixed $p$ (we will briefly describe the proof
of this fact in Subsection \ref{sec:conclude_top_down}). 
This explains why we need to consider $c_{p}$ with an increasing value of $p$. 

\section{Outline of proof}
\label{sec:outline}

In this section, we outline the proof of Theorem \ref{thm:main2}.
Assume that we want to prove that the edge-coloring of the complete
graph on the vertex set $\{0,1\}^{\alpha}$ given by $c_{p}$ is a
$(p+3,p+2)$-coloring. We will use induction on $\alpha$ to prove the stronger statement that the coloring is a $(q, q-1)$ coloring for all $q \leq p +3$.
To illustrate a simple case, assume that we are about to prove it
for $\alpha=r_{p+1}$ and have proved it for all smaller values of $\alpha$.
Let $S\subset\{0,1\}^{\alpha}$ be a given set of size at most $p+3$. We 
wish to show that the edges of $S$ receive at least $|S| - 1$ distinct colors.

Let $\alpha' = r_{p+1} - r_p$. For two vectors $v,w\in S$ satisfying $v\neq w$, let $v=(v',v'')$
and $w=(w',w'')$ where $v',w'\in\{0,1\}^{\alpha'}$ and 
$v'',w''\in \{0,1\}^{\alpha-\alpha'} = \{0,1\}^{r_{p}}$. 
Note that since $\alpha' = r_{p+1}-r_p$ is divisible by
$r_{p}$, the first $\frac{\alpha'}{r_p}$ blocks of resolution
$p$ of $v$ are identical to those of $v'$ and a similar fact
holds for $w$ and $w'$. 

If $v'=w'$ then, by the observation above, the first $\frac{\alpha'}{r_p}$
coordinates of $\xi_{p-1}$ are all zero. On the other hand, if $v'\neq w'$,
then the first block of resolution $p$ on which $v$ and
$w$ differ is one of the first $\frac{\alpha'}{r_p}$ blocks. Hence, in this case,
at least one of the first $\frac{\alpha'}{r_p}$ coordinates
of $\xi_{p-1}$ is non-zero. Thus, if we define sets $\Lambda_{I}$
and $\Lambda_{E}$ as 
\[
\Lambda_{I}=\big\{ c_{p}(v,w)\,:\, v'\neq w',\, v,w\in S\big\}
\]
and
\[
\Lambda_{E}=\big\{ c_{p}(v,w)\,:\, v'=w',\, v\neq w,\, v,w\in S\big\},
\]
then we have $\Lambda_{I}\cap\Lambda_{E}=\emptyset$. Hence, it suffices
to prove that $|\Lambda_{I}|+|\Lambda_{E}|\ge|S|-1$. The index `I'
stands for \emph{inherited colors} and `E' stands for \emph{emerging
colors}. 

The coloring $c_{p}$ contains more information
than necessary to prove that the number of colors is large. Hence, we consider only part of the coloring $c_{p}$. The part of 
the coloring that we consider for $\Lambda_{I}$ and $\Lambda_{E}$ will be different,
as we would like to highlight different aspects of our coloring
depending on the situation.

Define the sets $C_{I}$ and $C_{E}$ as
\[
C_{I}=\Big\{\big(c_{p}(v',w'),\eta_{p-1}(v'',w'')\big)\,:\, v' \neq w',\, v,w\in S\Big\}
\]
and 
\[
C_{E}=\Big\{\{v'',w''\}\,:\, v'=w',\, v'' \neq w'',\, v,w\in S\Big\}.
\]
We claim here without proof that $|C_{I}|\le|\Lambda_{I}|$ and $|C_{E}|\le|\Lambda_{E}|$.
Abusing notation,
for two vectors $v,w\in S$, we will from now on refer to the color
between $v$ and $w$ as the corresponding `color' in $C_{I}$ or
$C_{E}$. It now suffices to prove that $|C_{I}|+|C_{E}|\ge|S|-1$.

To analyze the colors in $C_{I}$ and $C_{E}$, we take a step back
and consider the first $\alpha'$ coordinates of the vectors in $S$.
Let $S'=\pi_{\alpha'}(S)$. Note that $S'$ is the collection of vectors
$v'$ in the notation above. There is a certain `branching
phenomenon' of vectors and colors. For a vector $v'\in S'$, let $T_{v'}=\{v\,:\,\pi_{\alpha'}(v)=v',\, v\in S\}$.
Hence, $T_{v'}$ is the set of vectors in $S$ whose first $\alpha'$
coordinates are equal to $v'$. Note that
\begin{equation}
\sum_{v'\in S'}|T_{v'}|=|S|.\label{eq:outline_1}
\end{equation}

Consider two vectors $v,w\in S$. If $v$ and $w$ are both in the same
set $T_{v'}$, then the color between $v$ and $w$ belongs to $C_{E}$
and if they are in different sets, then the color between $v$
and $w$ belongs to $C_{I}$. For a color $c\in C_{I}$, note that
the first coordinate of $c$ is of the form $c_{p}(v',w')$ for two
vectors $v',w'\in S'$. Further note that $c_{p}(v',w')$ is the color of an
edge that lies within $S'$. Hence, $c$ is a `branch' of some color
of an edge that lies within $S'$. In particular, by induction on
$\alpha$, we see that 
\begin{equation}
|C_{I}|\ge|S'|-1.\label{eq:outline_2}
\end{equation}

For a color $c\in C_{E}$, let $\mu_{c}$ be the number of (unordered)
pairs of vectors $v,w$ such that $c$ is the color between $v$ and
$w$. We have the following equation
\begin{equation}
\sum_{c\in C_{E}}\mu_{c}=\sum_{v'\in S'}{|T_{v'}| \choose 2}\ge\sum_{v'\in S'}(|T_{v'}|-1).\label{eq:outline_3}
\end{equation}

Let us first consider the simple case when $\mu_{c}=1$ for all $c\in C_{E}$
(that is, there are no overlaps between the emerging colors). In this
case, we have $|C_{E}|=\sum_{c\in C_{E}}\mu_{c}$. By (\ref{eq:outline_2}),
we have 
\begin{align*}
|C_{I}|+|C_{E}| & \ge(|S'|-1)+|C_{E}|=(|S'|-1)+\sum_{c\in C_{E}}\mu_{c},
\end{align*}
which by (\ref{eq:outline_3}) and (\ref{eq:outline_1}) is at least
\[
(|S'|-1)+\sum_{v'\in S'}(|T_{v'}|-1)=\big(\sum_{v'\in S'}|T_{v'}|\big)-1=|S|-1
\]
and thus the conclusion follows for the case when $\mu_c =1$ for all $c \in C_E$.

However, there might be some overlap between the emerging
colors. Note that there are $|C_{E}|$ emerging colors instead of the $\sum_{c\in C_{E}}\mu_{c}$
which we obtain by counting with multiplicity. Thus, there are $\sum_{c\in C_{E}}(\mu_{c}-1)$
`lost' emerging colors. Our key lemma asserts that every lost emerging
color will be accounted for by contributions towards $|C_{I}|$. Formally,
we will improve \eqref{eq:outline_2} and obtain the following inequality
\begin{equation}
|C_{I}|\ge(|S'|-1)+\sum_{c\in C_{E}}(\mu_{c}-1).\label{eq:outline_4}
\end{equation}
Given this inequality, we will have 
\[
|C_{I}|+|C_{E}|\ge(|S'|-1)+\sum_{c\in C_{E}}(\mu_{c}-1)+|C_{E}|=(|S'|-1)+\sum_{c\in C_{E}}\mu_{c},
\]
which, as above, implies that $|C_{I}|+|C_{E}|\ge|S|-1$.

\medskip{}

We conclude this section with a sketch of the proof of (\ref{eq:outline_4}).
To see this, we further study the branching of the colors. Define
$C_{B}$ as the set of colors that appear within the set $S'$, that is,
\[
C_{B}=\big\{ c_{p}(v',w')\,:\, v',w'\in S'\big\},
\]
where the index `B' stands for \emph{base colors}.
Every color $c\in C_{I}$ is of the form $c=(c',?)$, where $c'\in C_{B}$ and the question mark
`?' stands for an unspecified coordinate.
Thus, we immediately have at least $|C_{B}|$ colors in $C_{I}$ (this is the content
of Equation \eqref{eq:outline_2}).
Now take a color $c''=\{v'',w''\}\in C_{E}$
and suppose that $c''$ has multiplicity $\mu_{c''}$. Then there exist
vectors $x_{i}\in S'$ for $i=1,2,\ldots,\mu_{c''}$ such that $c''$
is the color between $(x_{i},v'')$ and $(x_{i},w'')$. Consider
the colors of the two pairs $\big((x_{1},v''),(x_{2},v'')\big)$
and $\big((x_{1},v''),(x_{2},w'')\big)$ in $C_I$. These are 
\begin{align*}
\big(c_{p}(x_{1},x_{2}),\eta_{p-1}(v'',v'')\big) & =(c_{1,2},0)\in C_{I} \quad\textrm{and}\\
\big(c_{p}(x_{1},x_{2}),\eta_{p-1}(v'',w'')\big) & =\big(c_{1,2},\eta_{p-1}(c'')\big)\in C_{I},
\end{align*}
respectively, where $c_{1,2}\in C_{B}$ (here we abuse notation and 
define $\eta_{p-1}(c'')=\eta_{p-1}(v'',w'')$, which is allowed since
the right-hand-side is symmetric in the two input coordinates). 
Note that by the inductive hypothesis,
there are at least $\mu_{c''}-1$ distinct colors of the form $c_{i,j}$
for distinct pairs of indices $i$ and $j$. Hence, by considering
these colors, we add colors of the types $(c_{i,j},0)$ and $(c_{i,j},\eta_{p-1}(c''))$
for at least $\mu_{c''}-1$ distinct colors $c_{i,j}\in C_{B}$. Even if one
of these two colors equals the color $(c_{i,j},?)$
counted above, we have added at least $\mu_{c''}-1$ colors to $C_{I}$ by
considering the color $c''\in C_{E}$.

Now consider another color $c_{1}''\in C_{E}$. This color
adds a further $\mu_{c_{1}''}-1$
colors to $C_{I}$ as long as $\eta_{p-1}(c'')\neq\eta_{p-1}(c_{1}'')$.
Therefore, if we can somehow guarantee that $\eta_{p-1}(c'')$ is distinct
for all $c''$, then we have 
\[
|C_{I}|\ge|C_{B}|+\sum_{c\in C_{E}}(\mu_{c}-1),
\]
which proves (\ref{eq:outline_4}), since $|C_{B}|\ge|S'|-1$ by the
inductive hypothesis.

Hence, it would be helpful to have distinct
$\eta_{p-1}(c'')$ for each $c''\in C_{E}$. Even though we cannot
always guarantee this, we can show that there exists a resolution in which
the corresponding fact does hold. This will be explained in more detail in Section \ref{sect:proofmainthm}.

\section{Properties of the coloring}
\label{sec:properties}

In this section, we collect some useful facts about the coloring functions
$c_{d}$. Before listing these properties,
we introduce the formal framework that we will use to describe
them.

\subsection{Refinement of functions\label{sub:refinement}}

For a function $f\,:\, A\rightarrow B$, let $\Pi_{f}=\{f^{-1}(b)\,:\, b\in f(A)\}$. Thus, $\Pi_{f}$ is a partition of $A$ into sets whose elements map by $f$ to the same element in $B$. For two functions $f$ and $g$ defined over the same domain, we say
that $f$ \emph{refines} $g$ if $\Pi_{f}$ is a refinement of $\Pi_{g}$.
This definition is equivalent to saying that $f(a)=f(a')$ implies
that $g(a)=g(a')$ and is also equivalent to saying that there exists
a function $h$ for which $g=h\circ f$. The term $f$ refines\emph{
}$g$ is also referred to as $g$ \emph{factors through} $f$ in category
theory. This formalizes the concept that $f$ contains
more information than $g$.

For two functions $f$ and $g$ defined over the same domain $A$,
let $f\times g$ be the function defined over $A$ where $(f\times g)(a)=(f(a),g(a))$.
The following proposition collects several basic properties of refinements
of functions which will be useful in the proof of the main theorem. 

\begin{prop}
\label{prop:refinement}Let $f_{1},f_{2},f_{3}$ and $f_{4}$ be functions defined
over the domain $A$.

(i) \emph{(Identity)} $f_1$ refines $f_1$.

(ii) \emph{(Transitivity)} If $f_{1}$ refines $f_{2}$ and $f_{2}$
refines $f_{3}$, then $f_{1}$ refines $f_{3}$. 

(iii) If $f_1$ refines $f_3$, 
then $f_{1}\times f_{2}$ refines $f_{3}$. 

(iv) If $f_1$ refines both $f_2$ and $f_3$, 
then $f_{1}$ refines $f_{2}\times f_{3}$. 

(v) If $f_1$ refines $f_3$ and $f_2$ refines $f_4$, 
then $f_{1}\times f_{2}$ refines $f_{3}\times f_{4}$. 

(vi) If $f_{1}$ refines $f_{2}$, then, for all $A'\subset A$,
we have $|f_{1}(A')|\ge|f_{2}(A')|$. \end{prop}

\begin{proof}
Let $\Pi_{i}=\Pi_{f_{i}}$ for $i=1,2,3$.

(i) This is trivial since $\Pi_{1}$ refines $\Pi_{1}$.

(ii) If $f_{1}$ refines $f_{2}$ and $f_{2}$ refines $f_{3}$, then
$\Pi_{1}$ refines $\Pi_{2}$ and $\Pi_{2}$ refines $\Pi_{3}$. Therefore,
$\Pi_{1}$ refines $\Pi_{3}$ and $f_{1}$ refines $f_{3}$.

(iii) Since $f_1 \times f_2$ clearly refines $f_1$, this follows from (ii).

(iv) If $f_{1}(a) = f_{1}(a')$, then $f_{2}(a)=f_{2}(a')$
and $f_{3}(a)=f_{3}(a')$. Hence, $(f_2 \times f_3)(a) = (f_2 \times f_3)(a')$ and
we conclude that $f_1$ refines $f_2 \times f_3$.

(v) By (iii), $f_1 \times f_2$ refines both $f_3$ and $f_4$. Therefore, by (iv),
$f_1 \times f_2$ refines $f_3 \times f_4$.

(vi) For $i=1,2$, let $\Pi_{i}\vert_{A'}=\{X\cap A'\,:\, X\in\Pi_{i},X\cap A'\neq\emptyset\}$
and note that $|f_{i}(A')|=\Big|\Pi_{i}\vert_{A'}\Big|$. Since
$\Pi_{1}$ is a refinement of $\Pi_{2}$, we see that $\Pi_{1}\vert_{A'}$
is a refinement of $\Pi_{2}\vert_{A'}$. Therefore, it follows that
\[
|f_{1}(A')|=\Big|\Pi_{1}\vert_{A'}\Big|\ge\Big|\Pi_{2}\vert_{A'}\Big|=|f_{2}(A')|,
\]
as required.
\end{proof}

Refinements arise in our proof because we often consider colorings with less information than the full coloring. In the outline above, we considered several different sets of colors, namely, $\Lambda_I$, $\Lambda_E$, $C_I$ and $C_E$ and we claimed without proof that $|C_I| \leq |\Lambda_I|$ and $|C_E| \leq |\Lambda_E|$. If we can show that $\Lambda_I$ is a refinement of $C_I$ and $\Lambda_E$ is a refinement of $C_E$, then these inequalities follow from Proposition \ref{prop:refinement} (vi) above.

\subsection{Properties of the coloring} 

%\begin{prop}
%\label{prop:diagonal} Let $\alpha \geq 1$ and $d \geq 0$ be integers.
%For all vectors $v,w\in\{0,1\}^{\alpha}$,
%\[
%c_{d}(v,v)=c_{d}(w,w)
%\]
%\end{prop}
%\begin{proof}
%Note that for all $d\ge0$, $\eta_{d}(v,v)=0$ and thus $\xi_{d}(v,v)=(0,0,\ldots,0)$
%(a vector of length $\lceil\frac{\alpha}{r_{d+1}}\rceil$, thus not
%depending on $v$). It follows that the value of $c_{d}(v,v)$
%does not depend on $v$.\end{proof}

%\begin{lem}
%\label{lem:monotone_c_d} For all integers $\alpha$, $d$ and $p$ with $\alpha \geq 1$ and $0 \leq d\le p$, $c_{p}$ refines
%$c_{d}$ (where both are considered as functions defined over $\{0,1\}^{\alpha}\times\{0,1\}^{\alpha}$).
%\end{lem}
%\begin{proof}
%The definition of the coloring implies that 
%$c_{p}=\xi_{p}\times \cdots \times \xi_{d+1}\times c_{d}$. 
%The conclusion then follows from Proposition \ref{prop:refinement}(i) and (iii).
%\end{proof}

We developed our formal framework for a rigorous 
treatment of the following two lemmas. It may be helpful at this stage to recall the definitions of $\eta_d$, $\xi_d$ and $c_{d}$ from Subsection \ref{subsect:defcolor}.

\begin{lem}
\label{lem:monotone_proj} Suppose that $\alpha$, $\alpha'$ and $d$
are integers with $d \geq 0$ and $1 \leq \alpha'\le\alpha$. Then the following hold
(where all functions are considered as defined over $\{0,1\}^{\alpha}\times\{0,1\}^{\alpha}$):

(i) $\eta_{d}$ refines $\eta_{d}\circ(\pi_{\alpha'}\times\pi_{\alpha'})$. 

(ii) $\xi_{d}$ refines $\xi_{d}\circ(\pi_{\alpha'}\times\pi_{\alpha'})$.

(iii) $c_{d}$ refines $c_{d}\circ(\pi_{\alpha'}\times\pi_{\alpha'})$.
\end{lem}
\begin{proof}
The case $\alpha'=\alpha$ is trivial so we assume that $\alpha'<\alpha$.

(i) Let $v$ and $w$ be vectors in $\{0,1\}^{\alpha}$ and let $v'=\pi_{\alpha'}(v)$
and $w'=\pi_{\alpha'}(w)$. We will show that one can compute the
value of $\eta_{d}(v',w')$ based only on the value of $\eta_{d}(v,w)$. This clearly implies the desired conclusion.

If $\eta_{d}(v,w)=0$, then $v=w$ and it follows that $\eta_{d}(v',w')=0$.
Assume then that $\eta_{d}(v,w)=(i,\{v_{i}^{(d)},w_{i}^{(d)}\})$ for
some index $i$ and blocks $v_{i}^{(d)},w_{i}^{(d)}$ of resolution
$d$. Let $j$ be the first coordinate in which the two vectors $v_{i}^{(d)}$
and $w_{i}^{(d)}$ differ. Then the first coordinate $x$ (note that
$1\le x\le\alpha$) in which $v$ and $w$ differ is $x=(i-1)\cdot r_{d}+j$
and satisfies 
\[
(i-1)\cdot r_{d}<x\le\min\{i\cdot r_{d},\alpha\}.
\]
Note that the values of $i$ and $j$ can be deduced from $\eta_{d}(v,w)$ and
hence $x$ can as well. It thus suffices to verify that $\eta_{d}(v',w')$ can be computed
using only $\alpha, \alpha', r_d$, $x$, $i$, $v_{i}^{(d)}$ and $w_{i}^{(d)}$.

If $\alpha'>i\cdot r_{d}$, then we have $\eta_{d}(v',w')=\eta_{d}(v,w)=(i,\{v_i^{(d)}, w_i^{(d)}\})$ and the claim is true.
On the other hand, if $\alpha'\le i\cdot r_{d}$, then there are
two cases. If $\alpha'<x$, then we have $v'=w'$. Therefore, $\eta_{d}(v',w')=0$
and the claim holds for this case as well. The final case is when $x\le\alpha'\le i\cdot r_{d}$.
In this case, we see that 
\[
\eta_{d}(v',w')=\left(i, \Big\{ \pi_{[\alpha'-(i-1)r_{d}]}(v_{i}^{(d)}),\pi_{[\alpha'-(i-1)r_{d}]}(w_{i}^{(d)}) \Big\} \right)
\]
and the claim holds. 

(ii) Let $v$ and $w$ be two vectors in $\{0,1\}^{\alpha}$. Then
\[
\xi_{d}(v,w)=\big(\eta_{d}(v_{1}^{(d+1)},w_{1}^{(d+1)}),\eta_{d}(v_{2}^{(d+1)},w_{2}^{(d+1)}),\ldots,\eta_{d}(v_{a+1}^{(d+1)},w_{a+1}^{(d+1)})\big),
\]
for some integer $a\ge0$. Let $v'=\pi_{\alpha'}(v)$ and $w'=\pi_{\alpha'}(w)$.
Suppose that $(j-1)r_{d+1}<\alpha'\le jr_{d+1}$. Then note that the
$j$-th block of resolution $d+1$ of $v'$ is $\pi_{[\alpha'-(j-1)r_{d+1}]}(v_{j}^{(d+1)})$
and that of $w'$ is $\pi_{[\alpha'-(j-1)r_{d+1}]}(w_{j}^{(d+1)})$. Then
$\xi_{d}(v',w')$ consists of $j$ coordinates, where for $1\le i<j$
the $i$-th coordinate is identical to the $i$-th coordinate of $\xi_{d}(v,w)$
and, for $i=j$, the $j$-th coordinate is 
\[
\eta_{d}\circ(\pi_{[\alpha'-(j-1)r_{d+1}]}\times\pi_{[\alpha'-(j-1)r_{d+1}]})(v_{j}^{(d+1)},w_{j}^{(d+1)}).
\]
Thus the function $\xi_{d}$ refines $\xi_{d}\circ(\pi_{\alpha'}\times\pi_{\alpha'})$
coordinate by coordinate (by part (i) of this lemma). Hence, by Proposition
\ref{prop:refinement}(v), we see that $\xi_{d}$ refines
$\xi_{d}\circ(\pi_{\alpha'}\times\pi_{\alpha'})$.

(iii) This follows from $c_{d}=\xi_{d}\times\cdots\times\xi_{0}$,
part (ii) of this lemma and Proposition \ref{prop:refinement}(v).
\end{proof}

Lemma \ref{lem:monotone_proj} seems intuitively obvious and might even seem 
trivial at first sight, but a moment's thought reveals the fact that it is
nontrivial. To see this, consider the function
\[ h_d(v,w) = \{v_i^{(d)}, w_i^{(d)}\}, \]
which is the projection to the second coordinate of $\eta_d(v,w)$.
Then the function $h_d$ fails to satisfy Lemma \ref{lem:monotone_proj}(i).
Moreover, if the functions $\xi_d$ and $c_d$ were built using $h_d$ instead of $\eta_d$, 
these would also fail to satisfy the claim of Lemma \ref{lem:monotone_proj}.

The next lemma completes the proof of one of the promised claims, namely, that $\Lambda_I$ (or, rather, a generalization thereof) refines $C_I$.
%This will be used to extract the most significant information 
%out of the coloring $c_p$.

\begin{lem}
\label{lem:transform_color} Suppose that positive integers $d,p,\alpha$
and $\alpha'$ are given such that $1\le d\le p+1$ and $\alpha'$
is the maximum integer less than $\alpha$ divisible by $r_{d}$.
Let $\gamma_{d}$ be the function which takes a pair of vectors $v,w\in\{0,1\}^{\alpha}$
as input and outputs 
\[
\gamma_{d}(v,w)=(c_{p}(v',w'),\eta_{d-1}(v'',w'')),
\]
where $v=(v',v'')$ and $w=(w',w'')$ for $v',w'\in\{0,1\}^{\alpha'}$
and $v'',w''\in\{0,1\}^{\alpha-\alpha'}$. Then $c_{p}\vert_{\{0,1\}^{\alpha}\times\{0,1\}^{\alpha}}$
refines $\gamma_{d}$.\end{lem}
\begin{proof}
For brevity, we restrict the functions to the set $\{0,1\}^{\alpha}\times\{0,1\}^{\alpha}$
throughout the proof. By Lemma \ref{lem:monotone_proj}(iii), we
know that $c_{p}$ refines $c_{p}\circ(\pi_{\alpha'}\times\pi_{\alpha'})$
and hence $c_{p}$ refines the first coordinate of $\gamma_{d}$.
On the other hand, since $\alpha'$ is the maximum integer less than
$\alpha$ divisible by $r_{d}$, the term $\eta_{d-1}(v'',w'')$ forms
the last coordinate of the vector $\xi_{d-1}(v,w)$. Hence, by Proposition
\ref{prop:refinement}(iii), $\xi_{d-1}$ refines $\eta_{d-1}(v'',w'')$.
By the definition of $c_{p}$ and Proposition \ref{prop:refinement}(iii),
we know that $c_{p}$ refines $\xi_{d-1}$. Therefore, by transitivity
(Proposition \ref{prop:refinement}(ii)), we see that $c_{p}$ refines
$\eta_{d-1}(v'',w'')$. Thus, $c_{p}$ refines both coordinates
of $\gamma_{d}$ and hence, by Proposition \ref{prop:refinement}(iv), 
we see that $c_{p}$ refines $\gamma_{d}$.
\end{proof}

\section{Proof of the main theorem}\label{sect:proofmainthm}

In this section we prove Theorem \ref{thm:main2}, which asserts that
for all $\alpha\ge1$ and $p\ge1$, the edge-coloring of the complete
graph on the vertex set $\{0,1\}^{\alpha}$ given by $c_{p}$ is a
$(p+3,p+2)$-coloring. We will prove by induction on $\alpha$ that every
set $S$ with $|S| \leq p+3$ receives at least $|S| - 1$ distinct colors.
The base case is when $\alpha\le r_{p}$. In this case, for two distinct
vectors $v,w\in\{0,1\}^{\alpha}$, we have $\xi_{p}(v,w)=\big(\eta_{p}(v,w)\big)=\big((1,\{v,w\})\big).$
Hence, for a given set $S\subset\{0,1\}^{\alpha}$, the edges within
this set are all colored with distinct colors, thereby implying that
at least ${|S| \choose 2}\ge|S|-1$ colors are used. 

Now suppose that $\alpha>r_{p}$ is given and
the claim has been proved for all smaller values of $\alpha$. Let
$S\subset\{0,1\}^{\alpha}$ be a given set with $|S|\le p+3$.
For each $1\le d\le p$, let $\alpha_{d}$ be the largest integer
less than $\alpha$ which is divisible by $r_{d}$. Note that since
$r_{d-1}$ divides $r_{d}$ for all $1\le d\le p$, we have 
\[
\alpha_{p}\le\alpha_{p-1}\le\cdots\le\alpha_{1}.
\]

For $1\le d\le p$, define sets $\Lambda_{I}^{(d)}$ and $\Lambda_{E}^{(d)}$
as 
\[
\Lambda_{I}^{(d)}=\big\{ c_{p}(v,w)\,:\,\pi_{\alpha_{d}}(v)\neq\pi_{\alpha_{d}}(w),\, v,w\in S\big\}
\]
and
\[
\Lambda_{E}^{(d)}=\big\{ c_{p}(v,w)\,:\,\pi_{\alpha_{d}}(v)=\pi_{\alpha_{d}}(w),\, v\neq w,\, v,w\in S\big\}.
\]
Since $\alpha_{d}$ is divisible by $r_{d}$, if $\pi_{\alpha_{d}}(v)=\pi_{\alpha_{d}}(w)$,
then the first $\mbox{\ensuremath{\frac{\alpha_{d}}{r_{d}}}}$ coordinates
of $\xi_{d-1}(v,w)$ will all be zero. On the other hand, if $\pi_{\alpha_{d}}(v)\neq\pi_{\alpha_{d}}(w)$,
then this is not the case. Since $\xi_{d-1}$ is part of $c_{p}$,
this implies that $\Lambda_{I}^{(d)}\cap\Lambda_{E}^{(d)}=\emptyset$.
Hence, for all $d$, the number of colors within $S$ is 
exactly $|\Lambda_{I}^{(d)}|+|\Lambda_{E}^{(d)}|$.
It therefore suffices to prove that $|\Lambda_{I}^{(d)}|+|\Lambda_{E}^{(d)}|\ge|S|-1$
for some index $d$.

We would like to extract only the important information from the
colors in $\Lambda_{I}^{(d)}$ and $\Lambda_{E}^{(d)}$. For each
$1\le d\le p$ and a given pair of vectors $v,w\in S$, let $v=(v_{d}',v_{d}'')$
and $w=(w_{d}',w_{d}'')$ for $v_{d}',w_{d}'\in\{0,1\}^{\alpha_{d}}$
and $v_{d}'',w_{d}''\in\{0,1\}^{\alpha-\alpha_{d}}$. Define the sets
$C_{I}^{(d)}$ and $C_{E}^{(d)}$ as
\[
C_{I}^{(d)}=\Big\{\big(c_{p}(v_{d}',w_{d}'),\eta_{d-1}(v_{d}'',w_{d}'')\big)\,:\, v_{d}' \neq w_{d}',\, v,w\in S\Big\}
\]
and 
\[
C_{E}^{(d)}=\Big\{\{v_{d}'',w_{d}''\}\,:\, v_{d}'=w_{d}',\, v_{d}'' \neq w_{d}'',\, v,w\in S\Big\}.
\]

By Lemma \ref{lem:transform_color} and Proposition \ref{prop:refinement}(vi),
we see that $|C_{I}^{(d)}|\le|\Lambda_{I}^{(d)}|$. We also have 
$|C_{E}^{(d)}|\le|\Lambda_{E}^{(d)}|$. To see this,
suppose that a color $\{v_{d}'',w_{d}''\}\in C_{E}^{(d)}$ comes from
a pair of vectors $v=(v_{d}',v_{d}'')$ and $w=(w_{d}',w_{d}'')$
in $S$. Since $v_{d}'=w_{d}'$ and $\alpha_{d}$ is divisible by
$r_{d}$, the function $\eta_{d}$ applied to the last pair of blocks
of resolution $d+1$ of $v$ and $w$ is equal to $(i,\{v_{d}'',w_{d}''\})$
for some integer $i$.
Therefore, the last coordinate of $\xi_{d}(v,w)$ has value $(i, \{v_{d}'',w_{d}''\})$.
This implies that $|C_{E}^{(d)}|\le|\Lambda_{E}^{(d)}|$. Hence, it
now suffices to prove that $|C_{I}^{(d)}|+|C_{E}^{(d)}|\ge|S|-1$
for some index $1\le d\le p$.

Assume for the sake of contradiction that we have $|C_{I}^{(d)}|+|C_{E}^{(d)}|\le|S|-2$
for all $1\le d\le p$. The following is the key ingredient in our
proof.
\begin{claim}
\label{claim:index}
If $|C_{I}^{(p)}|+|C_{E}^{(p)}|\le|S|-2$, then
there exists an index $d$ such that $\eta_{d-1}(c)$
is distinct for each $c\in C_{E}^{(d)}$.
\end{claim}
The proof of this claim will be given later. Let $d$ be the index
guaranteed by this claim and let $C_{I}=C_{I}^{(d)},\, C_{E}=C_{E}^{(d)}$.
Abusing notation,
for two vectors $v,w\in S$, we will from now on refer to the color
between $v$ and $w$ as the corresponding `color' in $C_{I}$ or
$C_{E}$.

Let $S'=\pi_{\alpha_{d}}(S)$ and, for a vector $v'\in S'$, let $T_{v'}=\{v\,:\,\pi_{\alpha_{d}}(v)=v',\, v\in S\}$.
Note that the sets $T_{v'}$ form a partition of $S$. Therefore,
\begin{equation}
\sum_{v'\in S'}|T_{v'}|=|S|.\label{eq:proof_1}
\end{equation}

Let $C_{B}$ be the set of colors which appear within the set $S'$
under the coloring $c_{p}$. 
Since $S'\subset\{0,1\}^{\alpha_{d}}$
and $\alpha_d < \alpha$, the inductive hypothesis implies that 
\begin{equation}
|C_{B}|\ge|S'|-1.\label{eq:proof_3}
\end{equation}

For a color $c\in C_{E}$, let $\mu_{c}$ be the number of (unordered)
pairs of vectors $v,w$ such that $c$ is the color between $v$ and
$w$. Note that
\begin{equation}
\sum_{c\in C_{E}}\mu_{c}=\sum_{v'\in S'}{|T_{v'}| \choose 2}\ge\sum_{v'\in S'}(|T_{v'}|-1).\label{eq:proof_2}
\end{equation}

Together with the three equations above, the following
bound on $|C_I|$, whose proof we defer for a moment, yields a contradiction.
\begin{equation} \label{eq:proof_4}
|C_{I}|\ge|C_{B}|+\sum_{c\in C_{E}}(\mu_{c}-1).
\end{equation}
Indeed, if this inequality holds, then, by
\eqref{eq:proof_4}, \eqref{eq:proof_3} and \eqref{eq:proof_2}, respectively, we have
\begin{align*}
 |C_{I}| + |C_E| &\ge \left((|S'| - 1) + \sum_{c \in C_E} (\mu_c - 1) \right) + |C_E|
   =  (|S'| - 1) + \sum_{c \in C_E} \mu_c  \\
   &\ge  (|S'| - 1) + \sum_{v' \in S'} (|T_{v'}| - 1) 
   = \left(\sum_{v' \in S'} |T_{v'}| \right) - 1.
\end{align*}
By \eqref{eq:proof_1}, we see that the right hand side 
is equal to $|S| - 1$. Therefore, we obtain $|C_I| + |C_E| \ge |S| -1$, 
which contradicts the assumption that
$|C_I| + |C_E| \le |S| - 2$.

To prove \eqref{eq:proof_4}, we examine the interaction between
the three sets of colors $C_I$, $C_B$ and $C_E$. 
Note that each color $c\in C_{I}$
is of the form $c=(c',?)$ for some $c'\in C_{B}$, where the question mark
`?' stands for an unspecified coordinate. This fact already
gives the trivial bound $|C_{I}| \ge |C_{B}|$. To obtain \eqref{eq:proof_4},
we improve this inequality by considering the `?' part of the
color and its relation to colors in $C_E$.
Take a color $c''=\{v'',w''\}\in C_{E}$
and suppose that $c''$ has multiplicity $\mu_{c''} \ge 2$. Then there exist
vectors $x, y \in S'$ such that $(x,v''), (x, w'') \in T_{x}$
and $(y,v''), (y, w'') \in T_y$.  Consider
the color of the pairs $\big((x,v''),(y,v'')\big)$
and $\big((x,v''),(y,w'')\big)$ in $C_{I}$. These colors are of the form
\begin{align*}
\big(c_{p}(x,y),\eta_{d-1}(v'',v'')\big) & =(c_p(x,y),0)\in C_{I} \quad\textrm{and}\\
\big(c_{p}(x,y),\eta_{d-1}(v'',w'')\big) & =\big(c_p(x,y),\eta_{d-1}(c'')\big)\in C_{I}.
\end{align*}
Here we abuse notation and define $\eta_{d-1}(c'')=\eta_{d-1}(v'',w'')$, which is allowed since
the right-hand-side is symmetric in the two input coordinates.
Therefore, having a color $c''$ with $\mu_{c''} \ge 2$ already implies that
$|C_I| \ge |C_B| + 1$. We carefully analyze the gain 
coming from these pairs for each color in $C_E$. To this end,
for each $x \in S'$, we define
\[ C_{E, x} = \Big\{\{v'', w''\} \,:\, (x,v''), (x,w'') \in T_{x}, \,\, v'' \neq w'' \Big\}. \]

For each $c' \in C_B$, we will count the number of colors of the form
$(c', ?) \in C_I$. There are two cases.

\noindent \textbf{Case 1} : For all $x , y \in S'$ with $c_p(x,y) = c'$,  $C_{E,x} \cap C_{E, y} = \emptyset$. 

Apply the trivial bound asserting that there is at least
one color of the form $(c', ?)$ in $C_I$.

\noindent \textbf{Case 2} : There exists a pair $x , y \in S'$ with $c_p(x,y) = c'$ such that $C_{E,x} \cap C_{E, y} \neq \emptyset$. 

If we have $c'' \in C_{E,x} \cap C_{E, y}$
for some $x, y \in S'$ with $c_p(x,y) = c'$, then, by the observation above,
we have both $(c', 0)$ and $(c', \eta_{d-1}(c''))$ in $C_I$. This shows that
the number of colors in $C_I$ of the form $(c', ?)$ is at least
\[ \left|\{(c', 0)\} \cup \left\{(c', \eta_{d-1}(c'')) \,:\, \exists x,y \in S',\,
               c_p(x,y) = c',\,\, c'' \in C_{E,x} \cap C_{E, y} \right\} \right|. \]
By Claim \ref{sub:Proof-of-Claim}, the function $\eta_{d-1}$ is
injective on $C_E$ and thus the above number is equal to
\[ 1 + \left| \left\{ c'' \,:\, \exists x,y \in S',\,
               c_p(x,y) = c',\,\, c'' \in C_{E,x} \cap C_{E, y} \right\} \right|. \]

\medskip

By combining cases 1 and 2, we see that the number of colors in $C_I$ satisfies 
\begin{align*}
 |C_I| \ge& \,\, |C_B| + \sum_{c' \in C_B} \left| \left\{c'' \,:\, \exists x,y \in S',\,
               c_p(x,y) = c',\,\, c'' \in C_{E,x} \cap C_{E, y} \right\} \right| \\
 =\,\,&\,\, |C_B| + \sum_{c'' \in C_E} \left| \left\{c' \,:\, \exists x,y \in S',\,
               c_p(x,y) = c',\,\, c'' \in C_{E,x} \cap C_{E, y} \right\} \right|.
\end{align*}
For a fixed color $c'' \in C_E$, there are precisely $\mu_{c''}$ vectors $x \in S'$ for
which the color $c''$ is in $C_{E,x}$. Hence, by the induction 
hypothesis, for each fixed $c''$, we have
\[ \left| \left\{c' \,:\, \exists x,y \in S',\,
               c_p(x,y) = c',\,\, c'' \in C_{E,x} \cap C_{E, y} \right\} \right| 
   \ge \mu_{c''} - 1. \]
Thus we obtain
\[ |C_I| \ge |C_B| + \sum_{c'' \in C_E} (\mu_{c''} - 1), \]
which is \eqref{eq:proof_4}.

\subsection{Proof of Claim \ref{claim:index}\label{sub:Proof-of-Claim}}

Claim \ref{claim:index} asserts that there exists an index $d$ such
that $\eta_{d-1}(c)$ is distinct for each $c\in C_{E}^{(d)}$.

It will be useful to consider the function $h_{d}$, which is defined as follows: 
for distinct vectors $v$ and $w$, define
$$h_{d}(v,w)=\{v_{i}^{(d)},w_{i}^{(d)}\},$$
where $v_{i}^{(d)},w_{i}^{(d)}$ are the first pair of blocks of resolution
$d$ for which $v_{i}^{(d)}\neq w_{i}^{(d)}$. Also, define
$h_{d}(v,v)=0$ for all vectors $v$. Note that we can also define
$h_{d}$ over unordered pairs $\{v,w\}$ of vectors as
$h_{d}(\{v,w\}) = h_{d}(v,w)$, 
since $h_d(v,w) = h_d(w,v)$ for all pairs $v$ and $w$. 
Throughout the subsection, by abusing notation,
we will be applying $h_d$ to both ordered and unordered pairs
without further explanation.

Recall that $\eta_{d}(v,w)=(i,\{v_{i}^{(d)},w_{i}^{(d)}\})$
and $\eta_{d}(v,v)=0$ and, therefore, $\eta_{d}$ refines $h_{d}$
(both considered as functions over the domain $C_{E}^{(d)}$).
Hence, to prove the claim, it suffices to prove that $h_{d-1}(c)$
is distinct for each $c\in C_{E}^{(d)}$. Another important observation
is that for all $1\le d\le p$, we can redefine the sets $C_{E}^{(d)}$
as 
\[
C_{E}^{(d)}=\big\{ h_{d}(v,w)\,:\,\pi_{\alpha_{d}}(v)=\pi_{\alpha_{d}}(w),\, v\neq w,\, v,w\in S\big\}.
\]

We first prove that there is a certain monotonicity between the sets
$C_{E}^{(d)}$ for $1\le d\le p$.
\begin{claim} \label{claim:emerged_monotonicity}
For all $d$ satisfying $2\le d\le p$,
there exists an injective map $\jmath_{d}\,:\, C_{E}^{(d-1)}\rightarrow C_{E}^{(d)}$
which maps $\{x,y\}\in C_{E}^{(d-1)}$ to
\[
\jmath_{d}(x,y)=\Big\{(v,x),(v,y)\Big\}\in C_{E}^{(d)},
\]
for some vector $v\in\{0,1\}^{\alpha_{d-1}-\alpha_{d}}$ depending
on the color $\{x,y\}$. Furthermore, $h_{d-1}\circ\jmath_{d}$ is the
identity map on $C_{E}^{(d-1)}$.\end{claim}
\begin{proof}
Take a color $\{x,y\}\in C_{E}^{(d-1)}$ and assume that $\{x,y\}=h_{d-1}(v_{x},v_{y})$
for $v_{x},v_{y}\in S$. By the definition of $C_{E}^{(d-1)}$, we may take $v_{x}$
and $v_{y}$ of the form 
\[
v_{x}=(v_{0},x)\quad\textrm{and}\quad v_{y}=(v_{0},y),
\]
for some vector $v_{0}\in\{0,1\}^{\alpha_{d-1}}$.
Fix an arbitrary such pair $(v_{x},v_{y})$ for each $\{x,y\} \in C_{E}^{(d-1)}$. 

 Let $v_{0}=(v_{1},v_{2})$
for $v_{1}\in\{0,1\}^{\alpha_{d}}$ and $v_{2}\in\{0,1\}^{\alpha_{d-1}-\alpha_{d}}$.
Then $v_{x}=(v_{1},v_{2},x)$ and $v_{y}=(v_{1},v_{2},y)$. Since
\[
\pi_{\alpha_{d}}(v_{x})=v_{1}=\pi_{\alpha_{d}}(v_{y}),
\]
we see that
\[
h_{d}(v_{x},v_{y})=\Big\{(v_{2},x),(v_{2},y)\Big\}\in C_{E}^{(d)}.
\]
Define $\jmath_{d}(x,y)=h_{d}(v_{x},v_{y})$ and note that
the range of $\jmath_{d}$ is indeed $C_{E}^{(d)}$. Moreover, since $v_{2}$
is a vector of length $\alpha_{d-1}-\alpha_{d}$ which is divisible
by $r_{d-1}$, we see that 
\[ h_{d-1}(\jmath_{d}(x,y)) = h_{d-1}\Big( (v_2,x), (v_2,y) \Big) = \{x,y\}.\]
The claim follows.
\end{proof}
In particular, Claim \ref{claim:emerged_monotonicity} implies that
\[
|C_{E}^{(1)}|\le|C_{E}^{(2)}|\le\cdots\le|C_{E}^{(p)}|.
\]

If $|C_{E}^{(1)}|\le1$, then $d=1$ trivially satisfies
the required condition. Hence, we may assume that $|C_{E}^{(1)}|\ge2$.
On the other hand, recall that we are assuming that $|C_{I}^{(p)}|+|C_{E}^{(p)}|\le|S|-2\le p+1$.
If $|C_{I}^{(p)}|=0$, then there exists at most one element 
$v_p \in \pi_{\alpha_{p}}(S)$
and all elements of $S$ are of the form $(v_p,x)$ for some
$x \in \{0,1\}^{\alpha - \alpha_p}$. But then
\begin{align} \label{eq:proof_of_claim_slack}
|C_E^{(p)}| \ge {|S| \choose 2} \ge |S| - 1, 
\end{align}
contradicting our assumption.
Therefore, we may assume that $|C_{I}^{(p)}|\ge1$,
from which it follows that $|C_{E}^{(p)}|\le p$. Hence, 
\[
2\le|C_{E}^{(1)}|\le|C_{E}^{(2)}|\le\cdots\le|C_{E}^{(p)}|\le p.
\]
If $p=1$, this is impossible. If $p\ge2$, then, by the pigeonhole
principle, there exists an index $d$ such that $|C_{E}^{(d-1)}|=|C_{E}^{(d)}|$.
For this index, the map $\jmath_{d}$ defined in Claim \ref{claim:emerged_monotonicity}
becomes a bijection. Then, since $h_{d-1}\circ\jmath_{d}$ is the
identity map on $C_{E}^{(d-1)}$, we see that $h_{d-1}(c)$ are distinct
for all $c\in C_{E}^{(d)}$. This proves the claim.

\section{Concluding Remarks}
\label{sec:remarks}

\subsection{Better than $(p+3,p+2)$-coloring}

Let $r=\sqrt{\frac{p+4}{2}}$. We can in fact prove that $c_{p}$ is
a $\left(p+\lfloor r\rfloor+1,p+\lfloor r\rfloor\right)$-coloring.
This improvement comes from exploiting the slackness of the inequality
\eqref{eq:proof_of_claim_slack} used in Subsection \ref{sub:Proof-of-Claim}.
To see this, we replace the bound on $S$ by $|S|\le p+r+1$ in the
proof given above. Since we have already proved the result for $|S| \le p+3$,
we may assume that $|S| \ge p+4$.

%Since $\alpha_p \le \alpha_{p-1} \le \cdots \le \alpha_1$,
%Lemma \ref{lem:monotone_proj}(iii) and Proposition \ref{prop:refinement}(vi) imply that
%\[
%|C^{(p)}_{I}|\le|C^{(p-1)}_{I}|\le\cdots\le|C^{(1)}_{I}|.
%\]
If $|C^{(p)}_{I}|\ge r-1$, then we have 
\[
|C^{(p)}_{E}|\le|S|-2-|C^{(p)}_{I}|\le p
\]
and we can proceed as in the proof above.
We may therefore assume that $|C^{(p)}_{I}|<r-1$. Let $S_p =\pi_{\alpha_p}(S)$. Then, since 
\[
|S_p|-1\le|C^{(p)}_{I}|<r-1,
\]
we know that $|S_p|<r$. Since 
\[
\sum_{v\in S_p}|\pi_{\alpha_{p}}^{-1}(v)|=|S|,
\]
there exists a $v\in S_p$ such that $|\pi_{\alpha_{p}}^{-1}(v)|\ge\frac{|S|}{|S_{p}|}$.
Note that every pair of vectors $w_{1},w_{2}\in\pi_{\alpha_{p}}^{-1}(v)$
gives a distinct emerging color. Moreover, by the inductive hypothesis,
we have at least $|S_p|-1$ inherited colors. Hence, the total number
of colors in the coloring $c_{p}$ within the set $S$ is at least
\[
|S_p|-1+{|\pi_{\alpha_{p}}^{-1}(v)| \choose 2}\ge |S_p|-1+\frac{1}{2}\frac{|S|}{|S_{p}|}\left(\frac{|S|}{|S_{p}|}-1\right),
\]
which, since 
\[ |S_{p}|< r=\sqrt{\frac{p+4}{2}}\le\sqrt{\frac{|S|}{2}},\]
is minimized when $|S_{p}|$ is maximized.
Thus the number of colors within the set $S$ is at least 
\[
\sqrt{\frac{|S|}{2}}-1+|S|-\sqrt{\frac{|S|}{2}}=|S|-1.
\]
This concludes the proof.

\subsection{Using fewer colors}
\label{sec:conclude_fewer_colors}

Recall that the coloring $c_{p}$ was built from the functions 
\[
\eta_{d}(v,w)=\big(i,\{v_{i}^{(d)},w_{i}^{(d)}\}\big),
\]
where $i$ is the minimum index for which $v_{i}^{(d)}\neq w_{i}^{(d)}$.
The function $\eta_{d}$ can in fact be replaced by the function 
\[
h_{d}(v,w)=\big\{v_{i}^{(d)},w_{i}^{(d)}\big\}
\]
(note that this is the function used in Section \ref{sub:Proof-of-Claim}).
In other words, even if we replace all occurrences of $\eta_{d}$ with
$h_{d}$ in the definition of $c_{p}$, we can still show that $c_{p}$
is a $(p+3,p+2)$-coloring. Moreover, there exists a constant $a_p$ such that the coloring of the complete graph on $n$ vertices defined in this
way uses only 
\[
2^{a_{p} (\log n)^{1-1/(p+1)}}
\]
colors. That is, we gain a $\log\log n$ factor in the exponent
compared to Theorem \ref{thm:main1}. The tradeoff is that the proof is now
more complicated, the chief difficulty being to find an appropriate
analogue of Lemma \ref{lem:monotone_proj} which works when $\eta_d$ 
is replaced by $h_d$. 

\subsection{Top-down approach}
\label{sec:conclude_top_down}

There is another way to understand our coloring as a generalization
of Mubayi's coloring. Recall that Mubayi's coloring is given as follows:
for two vectors $v,w\in[m]^{t}$ satisfying $v=(v_{1},\ldots,v_{t})$
and $w=(w_{1},\ldots,w_{t})$, let 
\[
c(v,w)=\big(\{v_{i},w_{i}\},a_{1},a_{2},\ldots,a_{t}\big),
\]
where $i$ is the minimum index for which $v_{i}\neq w_{i}$ and
$a_{j}=0$ if $v_{j}=w_{j}$ and $a_{j}=1$ if $v_{j}\neq w_{j}$.

Suppose that we are given positive integers $t_{1}$ and $t_{2}$.
For two vectors $v,w\in[m]^{t_{1}t_{2}}$, let $v=(v_{1}^{(1)},\ldots,v_{t_{2}}^{(1)})$
and $w=(w_{1}^{(1)},\ldots,w_{t_{2}}^{(1)})$ for vectors $v_{i}^{(1)}\in[m]^{t_{1}}$
and $w_{i}^{(1)}\in[m]^{t_{1}}$. Define the coloring $c^{(2)}$ as
\[
c^{(2)}(v,w)=\big(\{v_{i}^{(1)},w_{i}^{(1)}\},c(v_{1}^{(1)},w_{1}^{(1)}),\ldots,c(v_{t_{2}}^{(1)},w_{t_{2}}^{(1)})\big),
\]
where $i$ is the minimum index for which $v_i^{(1)} \neq w_i^{(1)}$.

Note that this can also be understood as a variant of $c$, where
we record more information in the $(a_{1},\ldots,a_{t})$ part of
the vector (this is a `top-down' approach and the previous definition
is a `bottom-up' approach). The coloring $c^{(2)}$ is essentially
equivalent to $c_2$ defined in Section \ref{sec:conclude_fewer_colors} above and can be
further generalized to give a coloring corresponding to $c_p$ for $p \ge 3$.
However, the proof again becomes more technical for this choice of definition.

One advantage of defining the coloring using this top-down approach
is that it becomes easier to see why the coloring $c_p$ on $K_{n_2}$ contains the
coloring $c_p$ on $K_{n_1}$, where $n_1 < n_2$, as an induced coloring.
To see this in the example above,
suppose that $n_1 = m^{t_1 t_2}$ and $n_2 = n^{s_1 s_2}$ for
$m \le n$, $t_1 \le s_1$ and $t_2 \le s_2$. Then the natural injection
from $[m]$ to $[n]$ extends to an injection from $[m]^{t_1}$
to $[n]^{s_1}$ and then to an injection  
from $[m]^{t_1 t_2}$ to $[n]^{s_1 s_2}$. This injection shows that
the coloring $c^{(2)}$ on $K_{n_2}$ contains the coloring $c^{(2)}$ on 
$K_{n_1}$ as an induced coloring. As in Section \ref{sub:Mubayi}, it then follows that 
$c^{(2)}$ (and thus $c_2$) fails to be a $(q,q^{\varepsilon})$-coloring 
for large enough $q$.
Similarly, for all fixed $p \ge 3$, we can show that $c_p$ fails to be 
a $(q,q^{\varepsilon})$-coloring for large enough $q$.

\subsection{Stronger properties}

We can show (see \cite{CoFoLeSu}) that Mubayi's coloring, 
discussed in Section \ref{sub:Mubayi}, actually has
the following stronger property: for every pair of colors, the graph whose edge set is the union of these two color classes has chromatic number at most three
(previously, we only established the fact that the clique number is at most three).
We suspect that this property can be generalized.

\begin{ques}
Let $p \ge 4$ be an integer. 
Does there exist an edge-coloring of the complete graph $K_n$ with 
$n^{o(1)}$ colors such that the union of every $p-1$ color classes has
chromatic number at most $p$?
\end{ques}

We do not know whether our coloring has this property or not.

\subsection{Lower bound}

Some work has also been done on the lower bound for $f(n,p,p-1)$.
As mentioned in the introduction, for $p = 3$ it is known that
$c'\frac{\log n}{\log \log n} \le f(n,3,2) \le c \log n$. For $p = 4$, 
the gap between the lower and upper bounds is much wider.
The well-known bound $r_k(4) \le k^{ck}$
on the multicolor Ramsey number of $K_4$ translates to
$f(n,4,3) \ge c\frac{\log n}{\log \log n}$, while
Mubayi's coloring gives an upper bound of
$f(n,4,3) \le e^{c\sqrt{\log n}}$.
The lower bound has been improved, first by Kostochka and Mubayi \cite{KM08},
to $f(n,4,3) \ge c\frac{\log n}{\log \log \log n}$
and then, by Fox and Sudakov \cite{FS09}, to
$f(n,4,3) \ge c\log n$, which is the current best known bound.

For $p \ge 5$, we can obtain a similar lower bound from the following formula,
valid for all $p$ and $q$.
\begin{align} \label{eq:recursive_lower} 
f\Big(n f(n,p-1,q-1), p, q\Big) \ge f(n,p-1,q-1).
\end{align}

To prove this formula, put $N= nf(n,p-1,q-1)$ and consider an edge-coloring of $K_N$
with fewer than $f(n,p-1,q-1)$ colors. It suffices to show that there
exists a set of $p$ vertices which uses at most $q-1$ colors on its edges.
If $f(n,p-1,q-1)=1$, then the inequality above is trivially true. 
If not, then for a fixed vertex $v$, there exists 
a set $V$ of at least $\left\lceil \frac{N-1}{f(n,p-1,q-1) - 1} \right\rceil \ge n$ vertices adjacent 
to $v$ by the same color. Since the edges within the set $V$ 
are colored by fewer than $f(n,p-1,q-1)$ colors,
the definition of $f(n,p-1,q-1)$ implies that we can find a set $X$ of $p-1$ vertices
with at most $q-2$ colors used on its edges. 
It follows that the set $X \cup \{v\}$ 
is a set of $p$ vertices with at most $q-1$ colors used on its edges. 
The claim follows.

From \eqref{eq:recursive_lower} and the lower bound 
$f(n,4,3) \ge c \log n$, 
one can deduce that 
$$f(n,p,p-1) \ge (1 + o(1)) f(n,4,3) \ge (c + o(1)) \log n$$ 
for all $p \ge 5$. On the other hand,
since the best known upper bound on $f(n,p,p-1)$ is
$$f(n,p,p-1) \le 2^{16p (\log n)^{1 - 1/(p-2)} \log \log n},$$
the gap between the upper and lower bounds gets
wider as $p$ gets larger. It would be interesting to know whether either bound can be substantially improved. 
In particular, the following question seems important.

\begin{ques}
For $p \ge 5$, can we give better lower bounds on $f(n,p,p-1)$ than the one
which follows from $f(n,4,3)$? 
\end{ques}

\end{document}